\documentclass[11pt, a4paper]{amsart}
\usepackage{amsmath,amssymb,latexsym}

\newcommand{\re}{\text{\rm Re\,}}
\newcommand{\im}{\text{\rm Im\,}}
\newcommand{\bx}{\mathbf{X}}

\newcommand{\bd}{{\mathbb{D}}}

\newcommand{\br}{{\mathbb{R}}}  

\newcommand{\bz}{{\mathbb{Z}}}
\newcommand{\bc}{{\mathbb{C}}}

\newcommand{\bt}{{\mathbb{T}}}

\newcommand{\ca}{{\mathcal{A}}}
\newcommand{\cb}{{\mathcal{B}}}
\newcommand{\cc}{{\mathcal{C}}}
\newcommand{\cf}{{\mathcal{F}}}

\newcommand{\cl}{{\mathcal{L}}}
\newcommand{\cm}{{\mathcal{M}}}

\newcommand{\cw}{{\mathcal{W}}}
\newcommand{\cx}{{\mathcal{X}}}
\newcommand{\cz}{{\mathcal{Z}}}

\newcommand{\fc}{{\mathfrak{C}}}
\newcommand{\fw}{{\mathfrak{W}}}
\newcommand{\fu}{{\mathfrak{U}}}
\newcommand{\fv}{{\mathfrak{V}}}

\renewcommand{\a}{\alpha}
\renewcommand{\b}{\beta}
\renewcommand{\l}{\lambda}

\newcommand{\s}{\sigma}

\renewcommand{\r}{\rho}

\newcommand{\p}{\varphi}

\renewcommand{\t}{\theta}
\renewcommand{\d}{\delta}

\renewcommand{\o}{\omega}
\newcommand{\oo}{\Omega}

\newcommand{\z}{\zeta}
\newcommand{\pp}{\Phi}

\newcommand{\lds}{{L^2(d\s)}}
\newcommand{\hhh}[1]{\hat{\hat{#1}}}
\newcommand{\ttt}[1]{\tilde{\tilde{#1}}}
\newcommand{\ccc}[1]{\check{\check{#1}}}

\newcommand{\ovl}{\overline}

\newcommand{\ti}{\tilde}

\newcommand{\lt}{\left}
\newcommand{\rt}{\right}

\newcommand{\dsp}{\displaystyle}

\allowdisplaybreaks
\numberwithin{equation}{section}

\newtheorem{theorem}{Theorem}[section]
\newtheorem{proposition}[theorem]{Proposition}

\theoremstyle{definition}

\begin{document}

\title[Schur algorithm, ORFs, Gaussian processes and prediction]
{Multipoint Schur algorithm, II: generalized moment problems, Gaussian processes and prediction} 
\author[L. Baratchart, L. Golinskii, S. Kupin]{L. Baratchart, L. Golinskii, S. Kupin}

\address{2004 route des Lucioles - BP 93
FR-06902 Sophia Antipolis Cedex, France}
\email{Laurent.Baratchart@sophia.inria.fr}

\address{Mathematics Division, Institute for Low Temperature Physics and
Engineering, 47 Lenin ave., Kharkov 61103, Ukraine}
\email{leonid.golinskii@gmail.com}

\address{IMB, Universit\'e Bordeaux 1, 351 cours de la Lib\'eration, 33405 Talence Cedex France}
\email{skupin@math.u-bordeaux1.fr}

\date{June, 21,  2010}

\thanks{This work was partially supported by grants  ANR-07-BLAN-024701 and  ANR-09-BLAN-005801}


\begin{abstract} 
We use nowdays classical theory of  generalized moment problems by Krein-Nudelman \cite{krnu} to define a special class of stochastic Gaussian processes. The class contains, of course, stationary Gaussian processes.  We obtain a spectral representation for the processes from this class and we solve the corresponding prediction problem. The orthogonal rational functions on the unit circle lead to a class of  Gaussian processes providing an example for the above construction.
\end{abstract}

\maketitle

\vspace{-0.5cm}
 
\section*{Introduction} \label{s0}

It is a classical fact that a (discrete-time) stationary Gaussian process admits a spectral representation which allows one to transfer  the study of various characteristics connected to the process to the  study of the shift operator on $L^2(d\mu)$, $\mu$ being a positive Borel measure on the unit circle. The measure $\mu$ is called the spectral measure of the process. The considerations pertaining to the geometry of the space $L^2(d\mu)$ provide striking applications of the theory of orthogonal polynomials on the unit circle, see Simon \cite{si1} for exhaustive treatment of the subject and its applications.

The purpose of the present paper is two-fold. Being interested in {\it orthogonal rational functions} on the unit circle (ORFs, for the sake of brevity), the authors wanted to understand what role these systems of functions play with respect to stochastic Gaussian processes. That is, is there a spectral representation of certain stochastic processes via ORFs? What kind of stochastic processes do arise in this way? How does one formulate (and solve) prediction problems for these processes? etc. A brief discussion of some aspects of this problem can be found in Dewilde-Dym \cite{de1}; the authors say that,  for a stationary Gaussian process,  one can apply the spectral representation theorem (see Theorem \ref{t03}). The past/future is given by subspaces looking like
$\cx^-_n=lin\, \{t^k : k\le n\}$ and $\cx^+_n=lin\, \{t^k: k\ge n+1\}$. Roughly speaking, they suggest however to consider "deformed" spaces $\cl_n$ (see Section \ref{ss01}), and, following the parallels with classical prediction problems, to study the projections to $\cl_n$ instead of $\cx^\pm_n$. The question is very natural from the analytical point of view and leads readily to the introduction and the study of ORFs. Nevertheless, the probabilistic  meaning of $\cl_n$ seems to be rather "opaque". One of questions we wanted to look at was to find a probabilistic interpretation for the described procedure.

Second, the book by Bultheel et al. \cite{bu1} treats, besides other interesting topics, the moment problems coming from ORFs (see Chapter 10). This made us think of a monograph by Krein-Nudelman \cite{krnu}, entirely devoted to generalized moment problems, descriptions of their solutions, the study of extremal solutions and different variations on the topic. The strong  appeal was to compare the results of both books and to attempt to "mix up" their approaches.

Following Krein-Nudelman, we introduce rather general class of moment problems and obtain the necessary and sufficient condition for their solvability (Section \ref{s1}). The questions on determinacy/indeterminacy of the problem, the description of solutions and the study of extremal solutions are left aside. Then we define a class of varying Gaussian processes and we obtain their spectral representation.  The term "varying" is intended to be synonymous to "generalized stationary", but seems to better reflect the situation. It means that the covariance matrices of the process change accordingly to a concrete rule determined by a sequence of external (and fixed) parameters. Corresponding prediction problems are formulated in a natural way; their solutions involve generalized orthogonal polynomials (or orthogonal Laurent polynomials). We specify the construction for the case of ORFs in Section \ref{s3}. This supplements \cite[Ch.~10]{bu1} with a solvability criterion (where the solvability conditions are not discussed) and provides an unified and natural approach to the problems of this type. 

\subsection{Generalized moment problems}\label{ss01}
Chapter 10 of the book by Bultheel et al. \cite{bu1} is devoted to the study of a special moment problem generated by the ORFs. The considered problem is a  particular case of  the concept of {\it a generalized moment problem} suggested by Krein-Nudelman in monograph \cite{krnu}. 
We adapt the latter point of view for our presentation since we believe that this "generalized" approach  allows one to gain in clarity as well as in generality. 

The content of this subsection is borrowed from the monograph by Krein-Nudelman \cite[Ch.~1-3]{krnu}. The definitions below are given for reader's convenience; we refer  to the book  \cite{krnu} for an extensive discussion and deep results on the subject. See also Grenander-Szeg\H o \cite{gre} for a classical presentation of a classical version of the topic.

All objects appearing in the first part of this subsection are real-valued. Let $\cm(\bt)$ be the set of positive finite Borel measures on $\bt$. Let $\fu=\{u_k\}_{k=0,\dots, \infty}$ be a system of continuous linearly independent real-valued functions defined on $\bt=\{|z|=1\}$ (or an interval of $\br$). Put $\cl_n=\cl_{n,\fu}$ to be the linear span of $u_0,\dots, u_n$. 

Let the system $\fu$ satisfy the following property: there are coefficients $\{a'_k\}_k, a'_k\in\br,$ such that
\begin{equation}\label{e01}
\sum^n_{k=0} a'_k u_k > 0
\end{equation}
on $\bt$. Let $\fc :\cl_n\to \br$  be the linear functional 
$$
\fc(\sum^n_k a_k u_k)=\sum^n_k a_k c_k,
$$
where $c_k=\fc(u_k)\in\br$.  The functional is called positive, if $\fc(f)\ge 0$ provided $f\in \cl_n, f\ge 0$ on $\bt$. The sequence $\{c_k\}_{k=0,\dots,\infty}$ is called positive (w.r.to $\fu$), if it defines a positive functional.  

The following theorem is of fundamental importance.
\begin{theorem}[{\cite{krnu}, Ch. 3, Theorem 1.1}]\label{t01} Let the system $\fu$ satisfy \eqref{e01}. The following assertions are equivalent:
\begin{itemize}
\item[-] the functional $\fc$ is positive (i.e., the sequence $\{c_k\}_k$ is positive),
\item[-] $\{c_k\}$ is a sequence of generalized moments, that is, there is a $\s\in \cm(\bt)$ such that
$$
c_k=\int_\bt u_k d\s.
$$
\end{itemize}
\end{theorem} 
The same result verbatim holds for infinite systems $\{u_k\}_{k=0,\dots,\infty}$ and $\{c_k\}_{k=0,\dots,\infty}$.

We now switch to complex-valued objects. Let $\fw=\{w_k\}_{k=0,\dots, \infty}$ be a system of continuous complex-valued functions on $\bt$. We require that $\{w_k\}_k\cup\{\bar w_k\}_k$ be linearly independent (w.r.to $\bc$). Setting $\fu=\{u_k\}_{k=0,\dots,\infty}$, $\fv=\{v_k\}_{k=0,\dots,\infty}$ with $u_k=\re w_k, v_k=\im v_k$, we rewrite the real-valued functional $\fc$ defined originally on $\cl_{n, \fu\cup\fv}$ as a complex-valued linear functional (denoted by the same letter) on $\cl_{n,\fw}$ defined as $\fc(u_k+i v_k)=\fc(u_k)+i\fc(v_k)$. It is convenient to put $w_{-k}=\bar w_k$ and one sees $\fc(\bar w_k)=\ovl{\fc(w_k)}$. 

From now on, $\cl_{n,\fw}$ is abbreviated as $\cl_n$. The functions  from $\cl_n$ might be thought of as "analytic polynomials" of degree $n$; the space $\bar\cl_n +\cl_n$ contains "trigonometric polynomials" of degree $n$.
\begin{theorem}[{\cite{krnu}, Ch.~3, Theorem 1.3}]\label{t02} Let the system $\fw$ satisfy
\begin{equation}\label{e02}
\sum^n_{k=0} (a_k w_k+\ovl{a_k w_k})>0
\end{equation}
for some coefficients $\{a_k\}, a_k\in \bc$. Then the following assertions for $\{c_k\}_k, c_k\in\bc,$ are equivalent:
\begin{itemize}
\item[-] The relation 
\begin{equation}\label{e03}
\sum^n_{k=0} (a_k w_k+\ovl{a_k w_k})\ge0
\end{equation}
implies that
$$
\sum^n_{k=0} (a_k c_k+\ovl{a_k c_k})\ge 0.
$$
\item[-] The coefficients $\{c_k\}_k$ are generalized moments (w.r.to $\fw$), that is: there exists a measure $\s\in\cm(\bt)$ with the property
$$
c_k=\int_\bt w_k d\s.
$$
\end{itemize}
\end{theorem}
The proof is by application of Theorem \ref{t01} to the real part of $\sum^n_{k=0} a_kw_k$.

\subsection{Spectral representation for stationary Gaussian processes}\label{ss02}
We recall some highly standard facts from the theory of stochastic processes; much more information on the topic is, for instance, in the monographs by Breiman \cite{bre} or Lamperti \cite{la}.

Let $\{\oo, \cf, P\}$ be a probability space and $\{X_n\}_{n\in\bz}, X_n:\oo\to \bc$ be a family of random variables.   The system $\bx=\{X_n\}_n$ is called a stochastic process. 
We assume that $E(X_n)=0, E(|X_n|^2)<\infty$.

Suppose now that $X_n$ have joint normal distribution; then $\{X_n\}$ is called a (stochastic) Gaussian process. For $\{X_{n_1},\dots, X_{n_m}\}$ define the covariance matrix by
$$
C=C_{n_1,\dots, n_m}=[c_{n_j n_k}], \quad c_{n_j n_k}=E(X_{n_k}\bar X_{n_j}),
$$
where $j,k=1,\dots, m$. The process is stationary if the covariance matrix is invariant w.r.to the shift, i.e.,
$$
C_{n_1,\dots, n_m}=C_{n_1+l,\dots, n_m+l}.
$$
A "stationary Gaussian process" will be often abbreviated as SGP or SG-process.

It follows from Herglotz' theorem (see \cite[Theorem 11.19]{bre}) that there is the unique $\mu\in\cm(\bt)$, called the spectral measure of the process, with the property $c_{jk}=\int_\bt t^{k-j} d\mu(t)$. 

The following theorem uses the concept of stochastic integral (see \cite[Ch.~11, Sect.~6]{bre} once again) and it gives the spectral representation of a SGP.

\begin{theorem}[{\cite{bre}, Theorem 11.21}]\label{t03} 
Let $\{X_n\}_n$ be a SGP and $\mu$ be its spectral measure. Then there exists a unique family of random variables $Z_\xi=Z( . , \xi), \xi\in\bt$ (i.e., $Z=Z(\o,\xi)$ with $\o\in\oo,\xi\in\bt$) having the properties:
\begin{enumerate}
\item For $\{\xi_1,\dots,\xi_k\}, \xi_j\in\bt, \xi_i\not =\xi_j$, the random variables $\{Z_{\xi_j}\}_j$ are jointly normally distributed.
\item For $I=[\xi_1,\xi_2)\subset \bt$, one writes $Z(I)=Z_{\xi_2}-Z_{\xi_1}$ and
$$
E(|Z(I)|^2)=\mu(I),\quad E(Z(I_1)\bar Z(I_2))=0, 
$$
for $I_1\cap I_2=\emptyset$.
\item Finally,
$$
X_n=\int_\bt \xi^n dZ( . , \xi).
$$
\end{enumerate}
\end{theorem}
Denote by $L^2(\bx)$ the closed linear span of $\bx$; this is a  subspace of $L^2(dP)$. Moreover, 
$L^2(\bx)$ is a Hilbert space with the scalar product $(X,Y)=E(X\bar Y)$. The theorem says that the map $U: L^2(\bx) \to L^2(d\mu)$ defined by $U(X_n)=\xi^n$ is a unitary one. If $\cz:L^2(\bx)\to L^2(\bx)$ acts as $\cz X_n=X_{n+1}$, its image under conjugation by $U$ (denoted by the same letter) is the usual shift, $\cz f(t)=t f(t), f\in L^2(d\mu)$.

It is not difficult to guess that the classical orthogonal polynomials w.r.to $\mu$ appear in this framework as solutions of forward/backward prediction problems, see \cite[Ch.~12]{bu1}. One of the purposes of this paper is to show that the ORFs play a similar role for a class of {\it  varying Gaussian processes}, see Section \ref{s2}.

\section{Generalized moment problem revisited}\label{s1}

We specify a little the construction from Section \ref{ss01} to fit our needs.  The essential part of the business is in \cite[Ch.~1,3]{krnu}.
Let $\fw=\{w_k\}_k$ be a system of functions described above and satisfying \eqref{e02}. Suppose that $\fw$ has few additional properties: 
\begin{enumerate}
\item $w_0=1$, \label{p1}
\item One has $\bar w_j w_k\in \bar\cl_j+\cl_k$, or, equivalently,
\begin{equation}\label{e1}
\bar w_j w_k=\sum^k_{s=-j} \b_{jk, s} w_s,
\end{equation}
where $\{\b_{jk, s}\}_s$ are some coefficients and, as before, $w_{-s}=\bar w_s$.
\item We have the following factorization property: if
$$
\sum^n_{k=0}(a_k w_k+ \ovl{a_kw_k})\ge 0
$$
on $\bt$, then necessarily
\begin{equation}\label{e2}
\sum^n_{k=0}(a_k w_k+ \ovl{a_kw_k})=\lt|\sum^n_{k=0} b_k w_k\rt|^2
\end{equation}
for some $\{b_k\}$. 
\end{enumerate}
Notice that the coefficients $\{\b_{jk, s}\}$ are uniquely determined by the system $\fw$. 

For a $\s\in\cm(\bt)$, the matrix $C_n=C_{n,\fw,\s}=[c_{jk}]_{j,k=0,\dots, n}$ is defined as 
\begin{equation}\label{e31}
c_{jk}=\int_\bt \bar w_j w_k d\s.
\end{equation}
Since $w_0=1$, one has  $c_{0k}=c_k$ and, moreover,
\begin{equation}\label{e3}
c_{jk}=\sum^k_{s=-j} \b_{jk, s} c_s,
\end{equation}
where, once again, $c_{-s}=\bar c_s$ and obviously $c_{jk}=\bar c_{kj}$. We say that a matrix $C_n=C_{n,\fw}=C_n^*\ge 0$ is a generalized Toeplitz matrix w.r.to $\fw$ ($\fw$-GTM, GTM or GT-matrix, for short), if its entries satisfy relations \eqref{e3}.  An interesting and important question is whether all $\fw$-GT-matrices come from a $\s\in \cm(\bt)$.

The answer (also explaining the terminology) is given by a slightly modified version of Theorem \ref{t02}. It says, in particular, that relations \eqref{e31} and \eqref{e3} define the same object. Everything below is "w.r.to $\fw$".

\begin{theorem}\label{t2} The following assertions are equivalent:
\begin{itemize}
\item[\it (1)] the GT-matrix $C=[c_{jk}]_{j, k=0,\dots, n}$ is positive, $C\ge 0$,
\item[\it (2)] the sequence $\{c_k\}_{k=0,\dots, n}$ is positive (in the sense of implication \eqref{e03}),
\item[\it (3)] there is a measure $\s\in\cm(\bt)$ with the property
$$
c_{jk}=\int_\bt \bar w_j w_k d\s.
$$
\item[\it (3')] one has
$$
c_k=\int_\bt w_k d\s.
$$
\end{itemize}
\end{theorem}

\begin{proof}
It is plain that (3) and (3') are equivalent. Indeed, (3) implies (3') since $c_k=c_{0k}=\int_\bt w_k d\s$. Conversely,
let $c_k=\int_\bt w_k d\s$ for any $k$. So, by \eqref{e3} and \eqref{e1},
$$
c_{jk}=\sum^k_{s=-j} \b_{jk, s} c_s= \int_\bt \sum^k_{s=-j} \b_{jk, s} w_s d\s=\int_\bt \bar w_j w_k d\s.
$$

Now, claims (2) and (3') are equivalent by Theorem \ref{t02}. Claim (3) yields (1) by the standard argument (see, for example, \cite[Ch. 3]{krnu}). The only implication to prove is  that (1) $\Rightarrow$ (3'), for instance.

By Theorem \ref{t02}, we have to prove that if $\sum^n_{k=0} (a_kw_k + \ovl{a_kw_k})\ge 0$ for some $\{a_k\}$, then necessarily $\sum^n_{k=0} (a_k c_k + \ovl{a_kc_k})\ge 0$.  By \eqref{e2}, one has
$$
\sum^n_{k=0} (a_kw_k + \ovl{a_kw_k})= \sum^n_{j,k=0} \bar b_j b_k \bar w_j w_k,
$$
and, consequently,
\begin{eqnarray*}
\sum^n_{k=0} (a_k c_k + \ovl{a_k c_k})&=& \sum^n_{k=0} (a_k T(w_k) + \bar a_k T(w_{-k}))\\
 &=&\sum^n_{j,k=0} \bar b_j b_k T(\bar w_j w_k)=\sum^n_{j,k=0} \bar b_j b_k \sum^k_{s=-j} \b_{jk, s} T(w_s) \\
 &=&\sum^n_{j,k=0} \bar b_j b_k \sum^k_{s=-j} \b_{jk, s} c_s=\sum^n_{j,k=0} \bar b_j b_k  c_{jk}\ge 0.
\end{eqnarray*}
The proof is complete.
\end{proof}

Let $\fw=\{w_k\}, \fw'=\{w'_k\}$ be two systems having the above properties plus that $\cl_n=lin\{w_k\}_{k=0,\dots, n}=lin\{w'_k\}_{k=0,\dots, n}$. Suppose also that
\begin{equation}\label{e30}
w'_k=\sum^{k}_{s=0}d_{sk} w_s,
\end{equation}
$d_{kk}\not =0$, so that the matrix $D=D_{\fw\to\fw'}=[d_{sk}]_{s,k=0,\dots, n}$ of the change of the basis is upper triangular and non-degenerate. The GT-matrices $C_{n,\fw'}$ and $C_{n,\fw}$ are connected in the obvious way
\begin{equation}\label{e301}
 C_{n,\fw'}=D^* C_{n, \fw} D.
\end{equation}
Certainly, this is an equivalence relation and we may get from $\fw'$ to $\fw$ if necessary.

\section{Varying Gaussian processes}\label{s2}
\subsection{Spectral representation for the process}\label{s21}
The point is to carry the construction of Section \ref{ss02} over the processes whose covariance matrices are generalized Toeplitz matrices in the sense of the previous section.

So, let $\fw$ be a fixed (infinite) system of functions. In addition to assumptions (1)-(3) on page \pageref{p1}, 
we suppose that the family $\fw$ is separating, that is:
\begin{enumerate}
\item[(4)] If, for a $\mu\in\cm(\bt)$, one has \label{p2}
$$
\int_\bt w_k d\mu =0
$$
for all $k$, then $\mu=0$.
\end{enumerate}
 
Before giving the definitions, let us discuss the terminology we use. Candidate labels for the objects introduced below, were "generalized stationary Gaussian processes" and "non-stationary Gaussian processes". It turned out, however, that these names led to some misunderstanding instead of clarifying the picture. For this reason, we stick with the name "a varying Gaussian process". This sentence means that the statistics (e.g., covariance matrices, etc.) of a Gaussian process under consideration vary accordingly to a prescribed law depending on a system of external parameters we may choose up to some extent. Of course, the class contains stationary Gaussian processes. We do not imply at all that the processes from the above class are non-stationary in a "wild" sense (i. e., the statistics of the process change absolutely arbitrarily with the shift of the index).

So, we say that $\bx$ is a varying Gaussian process (abbreviations: a $\fw$-VGP, a VGP, or a VG-process), if  $C=C_{n_1,\dots, n_m}=C(X_{n_1},\dots X_{n_m})$ is a generalized Toeplitz matrix w.r.to $\{w_{n_1},\dots, w_{n_m}\}$. We denote by $\mathbf{C}=\mathbf{C_X}=\{C_n\}_n$ the sequence of the covariance matrices of the process; they are $\fw$-GT matrices by definition. Theorem \ref{t2} implies that there exists the unique $\s\in\cm(\bt)$ such that
$$
C=C_{n_1,\dots, n_m}=\lt[\int_\bt \bar w_{n_j} w_{n_k} d\s\rt]_{j,k}.
$$
A counterpart to Theorem \ref{t03} in our situation is as follows.
\begin{theorem}\label{t3}
Let $\bx$ be a VGP and $\s$ be its spectral measure. Then there exists a unique family of random variables $Z_\xi=Z( . , \xi), \xi\in\bt$ with the properties:
\begin{enumerate}
\item For $\{\xi_1,\dots,\xi_k\}, \xi_j\in\bt, \xi_i\not =\xi_j$, the random variables $\{Z_{\xi_j}\}_j$ are jointly normally distributed.
\item For $I=[\xi_1,\xi_2)\subset \bt$, one writes $Z(I)=Z_{\xi_2}-Z_{\xi_1}$ and
$$
E(|Z(I)|^2)=\s(I),\quad E(Z(I_1)\bar Z(I_2))=0, 
$$
for $I_1\cap I_2=\emptyset$.
\item Finally,
$$
X_n=\int_\bt w_n(\xi) dZ( . , \xi).
$$
\end{enumerate}
\end{theorem}
As in Section \ref{ss02}, one has the unitary map $U:L^2(\bx)\to \cl=clos\, lin_{L^2(d\s)}\{w_k\}=L^2(d\s)$ (by the separation property (4) on page \pageref{p2}) and $U(X_n)=w_n$. The shift defined as $\cz X_n=X_{n+1}$ goes to $\cz(\sum_k a_k w_k)=\sum_k a_k w_{k+1}$; the sum of course is finite.

\begin{proof} The argument follows the proof of Theorem \ref{t03} (= \cite[Theorem 11.21]{bre}) with Bochner's theorem replaced by Theorem \ref{t02}. We give it for the completeness of the presentation.

Let $\s$ be the spectral measure of a $\fw$-VG-process. Consider $L^2(\bx)=clos\, lin_{L^2(dP)}\, \bx$; this is obviously a Hilbert space w.r.to the scalar product
$(X,Y)=E(X\bar Y)$.  Define a map $U:L^2(\bx)\to L^2(d\s)$ by the relation $U X_n=w_n$ and extend it by linearity to finite linear combinations of $\{X_n\}$. The map is obviously isometric, since for $Y_i=\sum_k a_{i,k} X_k$, and $f_i=\sum_k a_{i,k} w_k,\ i=1,2$, one has
\begin{eqnarray*}
(Y_1,Y_2)&=&E\lt(\sum_{k,j} a_{1,k}\bar a_{2,j}X_k\bar X_j\rt)\\
&=&\sum_{k,j} a_{1,k}\bar a_{2,j} E(X_k\bar X_j)=\sum_{k,j} a_{1,k}\bar a_{2,j} c_{jk}\\
&=&\sum_{k,j} a_{1,k}\bar a_{2,j} \int_\bt w_k\bar w_j d\s=(\sum_k a_{1,k} w_k,\sum_k a_{2,j} w_j)_\lds\\
&=&(f_1,f_2)_\lds.
\end{eqnarray*}
We extend the map $U$ by continuity to act from $L^2(\bx)$ to $L^2(d\s)$. It is easy to see that it is one-to-one and hence unitary. 

Denote the (clockwise) arc $[1,\xi)\subset\bt$ by $I_\xi$. Since $\chi_{I_\xi}\in\lds$, we set $Z_\xi=U^{-1}(\chi_{I_\xi})\in L^2(\bx)$, where $\chi_{I_\xi}$ is the indicator function of $I_\xi$. We now verify the properties of $Z_\xi$ claimed in the formulation of the theorem. Let us start with (1): let  
$\{Y^n\}=\{(Y^n_1,\dots, Y^n_m)\}_{n=0,1,\dots}$ be a sequence of vector-valued random variables. Assume that $(Y^n_1,\dots, Y^n_m)$ are jointly normally distributed and $Y^n\to Y=(Y_1,\dots, Y_m)$ in $L^2(\bx)$. It follows that $(Y_1,\dots, Y_m)$ are jointly normally distributed and $C(Y_j\bar Y_k)=\lim_{n\to\infty} C(Y^n_j\bar Y^n_k)$ (actually one has to argue for real-valued random variables and then pass to complex-valued ones separating real and imaginary parts).

As for (2), we have for arcs $I,I_1,I_2\subset \bt, I_1\cap I_2=\emptyset$
$$
E(|Z(I)|^2)=(\chi_{I},\chi_{I})_\lds=\s(I).
$$
and
$$
E(Z(I_1)\bar Z(I_2))=(\chi_{I_1},\chi_{I_2})_\lds=0,
$$

For (3), let $f$ be a continuous (and hence uniformly continuous) function on $\bt$. Take a partition $\bt$ in a family of left-closed, right-open disjoint intervals $\{I_k\}$. Set $\l_k\in I_k$ and notice that
$$
\sum_k f(\l_k)\chi_{I_k}=U\lt(\sum_k f(\l_k) Z(I_k)\rt).
$$
When $\max_k|I_k|\to 0\ (k\to\infty)$, the argument of $U$ goes to $\int_\bt f(\xi)dZ( . ,\xi)$ (in $L^2(dP)$), and the left-hand side of the equality goes uniformly to $f$. Hence
$$
U^{-1} f=\int_\bt f(\xi) dZ( . ,\xi).
$$
Writing this for $f=w_n$, we come to
$$
X_n=\int_\bt w_n(\xi) dZ(.,\xi).
$$
The theorem is proved.
\end{proof}
We continue with few remarks on the result. First, observe that we can prove a likewise result for stochastic (varying) non-Gaussian processes dropping the first point of the theorem; see Lamperti \cite{la} for the stationary case.

Second, since we use only $L^2$-convergence, the same proof goes through for a fixed $\s$ with the property
\begin{equation}\label{e33}
\sup_n \int_\bt |w_n|^2\, d\s<\infty.
\end{equation}
The condition means of course 
\begin{equation}\label{e330}
\sup_n E(|X_n|^2)<\infty.
\end{equation}
in terms of the process $\mathbf{X}$.

Third, suppose two systems $\fw$ and $\fw'$  are related by \eqref{e301}, but $\fw'$ does not necessarily have property \eqref{e33}. Of course, Theorem  \ref{t3} does not work directly for $\fw'$-VG-processes. Nevertheless, applying an appropriate (non-stationary) linear filter to a $\fw'$-VG-process $\mathbf{Y}$, we can get to a $\fw$-VG-process $\mathbf{X}$, obtain its spectral representation from Theorem \ref{t3}, and return back to the initial process $\mathbf{Y}$ using the inverse filter. In formulas: let $D=D_{\fw\to\fw'}$, $D'=D^{-1}=D_{\fw'\to\fw}$. Given a $\fw'$-VG-process $\mathbf{Y}$, define the filtered process $\mathbf{X}$ (compare to \eqref{e30})
\begin{equation}\label{e302}
 X_k=\sum^k_{s=0} d'_{sk}Y_s.
\end{equation}
An easy computation shows $C_{n,\mathbf{X}}=D'^*C_{n,\mathbf{Y}} D'$, and the identification with \eqref{e301} shows that $\mathbf{X}$ is a $\fw$-VG-process. By the spectral representation theorem (i.e., Theorem \ref{t3}) $X_k=\int_\bt w_k(\xi) dZ_{\mathbf{X}}(.,\xi)$. On the other hand,
\begin{equation}\label{e303}
 Y_k=\sum^k_{s=0} d_{sk}X_s,
\end{equation}
and
\begin{equation}\label{e304}
 Y_k=\int_\bt \lt(\sum^k_{s=0} d_{sk} w_s(\xi) \rt) dZ_{\mathbf{X}}(.,\xi)=\int_\bt w'_k(\xi) dZ_{\mathbf{X}}(.,\xi).
\end{equation}
Hence,  speaking a bit loosely, Theorem \ref{t3} can be extended to VG-processes which are "similar", in a proper sense, to processes satisfying its initial assumptions.

\subsection{Forward, backward and forward/backward prediction problems for VGPs}
\label{s22}
We start with a bit of terminology. Let $\fw=\{w_k\}_{k=0,\dots, \infty}$ and $\s\in\cm(\bt)$. Since the family $\fw$ is linearly independent on $\bt$, it is also linearly independent in $L^2(d\s)$. The (generalized) orthogonal polynomials w.r.t to $\fw$ are denoted by $\{\p_k\}, ||\p_k||_{L^2(d\s)}=1$. They are well-defined and obtained by the Gram-Schmidt orthonormalization procedure from $\{w_k\}_k$ in $L^2(d\s)$. The (generalized) reversed orthogonal polynomials $\{\p^*_k\}$ w.r.to $\fw$ (warning: the notation is natural but slightly abusive!) are defined by orthonormalization procedure to satisfy  $\p^*_k\in \cl_k, ||\p^*_k||_{L^2(d\s)}=1$, and $(\p^*_k, w_j)=0$ for $j=1,\dots,k$. Equivalently, on can get $\{\p^*_k\}$ orthonormalizing $\{w_k,w_{k-1},\dots, w_0\}$ and taking the last element in the obtained sequence. The corresponding monic (also abusive!) orthogonal polynomials are denoted by $\{\pp_k\},\ \{\pp^*_k\}$.

When $\fw=\{w_k\}_{k\in\bz}$ is a bilateral  system of functions, we speak about (generalized) orthogonal Laurent polynomials (GOLP, GOL-polynomials) $\{\chi_k\}_{k=0,\dots,\infty}$ (alternatively,  $\{x_k\}_{k=0,\dots,\infty}$) rather than "usual" ones. They are defined by the Gram-Schmidt procedure from the sequence $\{w_0, w_1, w_{-1},$ $w_2, w_{-2}, \dots\}$. The sequence $\{x_k\}_k$ comes from orthonormalization of   $\{w_0, w_{-1},$ $ w_1, w_{-2}, w_2,\dots\}$. 
The monic versions of these sequences will be denoted by $\{\chi^m_k\}_k$ and $\{x^m_k\}_k$.

We are interested in one step prediction problems, see Bultheel et al. \cite[Ch.~12]{bu1} for the details. 
To formulate the problem, let $\cx_{0,n-1}=lin \{X_k : k=0,\dots, n-1\}$ and $\cx_{1,n}=lin \{X_k : k=1,\dots, n\}$. The forward prediction problem is to compute $\hat X_n=(X_n| \cx_{0,n-1})=Pr_{\cx_{0,n-1}} X_n$ and the backward one targets  $\check X_0=(X_0| \cx_{1, n-1})=Pr_{\cx_{1,n}} X_0$. The projections are understood in the $L^2(dP)$-sense and $( . | \dots)$ is a probabilistic notation for the object. Recall that the projections can be interpreted as conditioned random variables. We also look at corresponding innovation processes, $\hat E_n=X_n-\hat X_n$ and $\check E_0=X_0-\check X_0$.

By the spectral representation theorem (i.e., Theorem \ref{t3}), this is equivalent to the computation of projections $\hat w_n=Pr_{\cw_{0,n-1}} w_n$ and $\check w_0=Pr_{\cw_{1,n}} w_0$, where
$\cw_{0,n-1}=lin \{w_k : k=0,\dots, n-1\}$ and $\cw_{1,n}=lin \{w_k : k=1,\dots, n\}$.  It is not difficult to see that, by definition, $\hat w_n=w_n-\pp_n$ and $\check w_0=w_0-\pp^*_n$ etc. Recalling that $X_n=\cz^n X_0$, we get back to the process $\bx$. This gives
\begin{eqnarray}
\hat X_n=(\cz^n-\pp_n(\cz))X_0, &&  \hat E_n=\pp_n(\cz)X_0, \label{e35}\\
\check X_0=(I-\pp^*_n(\cz))X_0, && \check E_0=\pp^*_n(\cz)X_0.\nonumber
\end{eqnarray}


By mixed prediction problem (or backward-forward prediction problem) we understand an estimate of the present from a part of the past plus a part of the future. The construction goes along the same lines as above, that is why we give the formulation of the problem followed by its solution. So, let $\cx_{-n,-1;1,n}=lin \{X_k : k=-n,\dots, n,\ k\not =0 \},  \cx_{-(n+1),-1;1,n}=lin \{X_k : k=-(n+1),\dots, n,\ k\not =0 \}$, and $\cx_{-n,-1;1,n+1}=lin \{X_k : k=-n,\dots, n+1,\ k\not =0 \}$. The different mixed prediction problems are: find
$\ttt{X_0}=(X_0| \cx_{-n,-1;1,n}), \hhh{X_0}=(X_0| \cx_{-n,-1;1,n+1})$, and $\ccc{X_0}=(X_0|\cx_{-(n+1),-1;1,n})$ and expressions for corresponding innovation processes. The first case can be treated with the help either $\{\chi_k\}_k$ or $\{x_k\}_k$, the second case suggests the use of $\{\chi_k\}_k$, and the third one - the application of $\{x_k\}_k$. We obtain
\begin{eqnarray*}
&& \ttt{X_0}=(I-\chi^{m\, *}_{2n})(\cz)X_0,\quad \ttt{E_0}=X_0-\ttt{X_0}=\chi^{m\, *}_{2n}(\cz)X_0, \\
&&\hhh{X_0}=(I-\chi^{m\, *}_{2n+1})(\cz)X_0, \quad  \hhh{E_0}=X_0-\hhh{X_0}=\chi^{m\, *}_{2n+1}(\cz)X_0, \\
&& \ccc{X_0}=(I-x^{m\, *}_{2n+1})(\cz)X_0, \quad  \ccc{E_0}=X_0-\ccc{X_0}=x^{m\, *}_{2n+1}(\cz)X_0,  
\end{eqnarray*}


\section{Orthogonal rational functions and a special class of moment problems  and VG-processes}\label{s3}
The point of this section is to illustrate the construction from Sections \ref{s1}, \ref{s2} with the help of the so-called orthogonal rational functions  (ORFs, to be brief) on the unit circle $\bt$. The monograph by Bultheel et al. \cite{bu1} gives the state-of-art of the subject. A shorter overview is in Bultheel et al. \cite{bu2}. The paper Baratchart et al. \cite{bar1} contains some recent advances on the topic; our notation follows the last paper.

\subsection{ORFs and generalized moment problems}\label{s31}
Let $\{\a_k\}_{k=0,\dots,\infty}, \a_0=0$ be a given sequence of  $\bd$ and,  as in \cite{bar1}, 
\begin{equation}\label{e40}
\sum_k (1-|\a_k|)=+\infty.
\end{equation}
We suppose for simplicity that $\a_k\not =\a_j, k\not=j$; it goes without saying that the general situation can be treated along the same lines.
Let 
$$
\zeta_k(t)=\frac{t-\a_k}{1-\bar \a_k t},\quad \cb_0=1,\ \cb_k=\prod^k_{j=1} \z_j, 
$$
for $k\ge 1$. Furthermore, let 
$$
\cl_n=lin\{\cb_k: k=0,\dots,n\}=lin\lt\{\frac 1{1-\bar\a_k t}: k=0,\dots,n\rt\}.
$$
We consider the following systems $\fw_i=\{w_{ik}\},\ i=1,2,3$:
\begin{enumerate}\label{p3}
\item $w_{10}=1,\  w_{1k}=\cb_k,\  k=1,\dots,n,$
\item  $\dsp w_{20}=1,\  w_{2k}=\frac 1{1-\bar \a_k t},\  k=1,\dots,n,$
\item   $w_{30}=1,\ w_{3k}=\frac{t^k}{\prod^k_{j=1} (1-\bar\a_j t)},\  k=1,\dots,n$.
\end{enumerate}
Obviously, every system $\fw_i$  forms a basis of $\cl_n$. The first system is natural in applications to VG-processes with the property $\sup_n E(|X_n|^2)<\infty$. The second one is from \cite[Ch.~3, Sect.~2]{krnu} and appears readily from interpolation (e.g., Schur-Nevanlinna-Pick) problems. One of its lacks though is that the condition \eqref{e33} for any $\s\in\cm(\bt)$ is satisfied iff the sequence $\{\a_k\}$ is compactly contained in $\bd$.
To fix this, one can consider the system $\fw'_2$
$$
w'_{20}=1,\  w'_{2k}=\frac{1-|\a_k|}{1-\bar \a_k t},\  k=1,\dots,n
$$
instead. The third system is from \cite[Ch.~9]{bu1} and it has nice invariance properties w.r.to $*$-operation.

As the following proposition shows, the systems are completely equivalent to our purposes.
\begin{proposition}\label{pr1} We have:
\begin{enumerate}
\item The change-of-basis matrices taking one $\fw_i$ to another are upper triangular and non-degenerate.
\item The systems $\fw_i$ satisfy assumptions (1)-(3) and (4) on pp.~\pageref{p1}, \pageref{p2}. 
\end{enumerate}
\end{proposition}
\begin{proof}
First, observe that 
\begin{equation}\label{e4}
w_{1k}=\cb_k=A_{0k}+ \sum^k_{j=1} \frac 1{\ovl{\cb_{k,\check j}(\a_j)}}\z_j,
\end{equation}
where $\cb_{k,\check j}=\prod^k_{s\not= j}\z_s$, and $A_{0k}$ is an easily computable constant. The proof is by inspection of residues of the LHS and the RHS at point $\{1/\bar\a_j\}$. We also have $\z_k=\frac 1{\bar\a_k}(\frac{\r_k}{1-\bar\a_k t}-1)$, where $\r_k=(1-|\a_k|^2)$ and $k\ge 1$. So, the first claim of the proposition is true for the passage from (2) to (1).

To go from (3) to (2), we write
$$
w_{3k}=B_{0k}+\sum^k_{j=1}\frac 1{\lt(\bar\a_j\, \prod^k_{s\not=j}(\bar \a_j-\bar \a_s)\rt)}\, \frac 1{1-\bar\a_j t},
$$ 
the proof being similar; $B_{0k}$ is a constant once again. The rest follows from the group properties of the square upper triangular matrices (with the non-degenerate diagonal).

This shows that it is enough to prove the second claim of the proposition for any of $\fw_i$, the rest will enjoy the same properties.  The key (and elementary) point is the following very well-known fact (see, for instance, \cite[Ch.~3, Sect.~2]{krnu}): a trigonometric polynomial is positive on $\bt$ iff  it can be represented as an analytic polynomial times its conjugate, i.e.,
\begin{equation}\label{e5}
0\le \sum^n_{j=0} a_jt^j +\ovl{a_j t^j}=\lt| \sum^n_{j=0} b_j t^j\rt|.
\end{equation}
We go through the (easy) proof of (2) of the current proposition for the system $\fw_1$. (1), p.~\pageref{p1} being clear, we notice that, for $j<k$
$$
\bar w_{1j} w_{1k}=\z_{j+1}\dots \z_{k},
$$
and relation (2), p.~\pageref{p1} follows from \eqref{e4}.  As for (3), p.~\pageref{p1}, let
$$
\sum^n_k (a_k\cb_k +\ovl{a_k\cb_k})\ge 0
$$
on $\bt$ for some $\{a_k\}$. We can rewrite the latter expression as $f(t)/\prod_k |1-\bar\a_k t|^2$, where $f\ge 0$ on $\bt$ is a positive trigonometric polynomial. Relation \eqref{e5} implies that $f=p\cdot\bar p$, where $p$ is an analytic polynomial, $\deg p\le n$. Hence,
$$
\sum^n_k (a_k\cb_k +\ovl{a_k\cb_k})=h\cdot \bar h,
$$
where $h=p/\prod_k (1-\bar\a_k t)$, which is definitely in $\cl_n$.

As for (4), p.~\pageref{p2}, we see that \eqref{e40} says that $lin \{\bar\cl_n +\cl_n : n\}$ is dense in $\ovl{\ca(\bt)}+\ca(\bt)$, which is, in turn, dense in $\cc(\bt)$. The separation property and the proposition follows.
\end{proof}

Suppose now that $\s\in\cm(\bt)$ and $c^i_k=c^i_{0k}=\int_\bt w_{ik} d\s$, or, more generally,
\begin{equation}\label{e51}
c^i_{jk}=\int_\bt \bar w_{ij} w_{ik} d\s,
\end{equation}
for $i=1,2,3$.
Of course $c^i_{jk}=\bar c^i_{kj}$. The corresponding generalized moment (Toeplitz) matrix looks like
$$
C^i_n=
\begin{bmatrix}
c^i_{00} & c^i_{01} & c^i_{02} & \dots & c^i_{0n}\\
c^i_{10} & c^i_{11} & c^i_{12} & \dots & c^i_{1n}\\
\vdots  & \vdots    & \ddots & \cdots & \vdots\\
 c^i_{n0} & \dots  & \dots & \dots & c^i_{nn}
 \end{bmatrix}\ge 0.
$$
As Proposition \ref{pr1} shows, the matrices $C^i_n$ are conjugated by non-degenerate upper triangular matrices.

From now on, we deal mostly with $C^1_n$ and $C^2_n$. The same development works for $C^3_n$ but is rather cumbersome, that is why we do not give it. Recall that $c^1_{kk}=c^1_{00}$.  Put $\cb_{jk}=\prod^k_{s=j+1}\z_s$ and notice that
\begin{equation}\label{e52}
c^1_{jk}=\int_\bt \bar w_{1j} w_{1k} d\s=\int_\bt \cb_{jk} d\s,
\end{equation}
and, with the help of  \eqref{e4},
$$
\cb_{jk}=A_{jk}+ \sum^k_{s=j+1} \frac 1{\ovl{\cb_{jk,\check s}(\a_s)}}\z_s,
$$
where $A_{jk}$ is a coefficient (which is easy to compute) and $\cb_{jk,\check s}$ stays for the product $\cb_{jk}$ with dropped $s$-th factor. So,
\begin{equation}\label{e6}
c^1_{jk}=A_{jk}+ \sum^k_{s=j+1} \frac 1{\ovl{\cb_{jk,\check s}(\a_s)}}c^1_s.
\end{equation}
Hence, if matrix $C^1_n$ is defined by relations \eqref{e51}, we have $C^1_n={C^1_n}^*\ge 0$  and it conforms to relations \eqref{e6}. It is natural to ask the inverse: is it true that a non-negative matrix $C^1_n$ satisfying relations \eqref{e6} can be represented as a Toeplitz matrix of a measure $\s\in\cm(\bt)$ (i.e., satisfies \eqref{e51})?

The answer is yes and it is given by a straightforward consequence of Theorem \ref{t2}.

\begin{proposition}\label{pr2} The following assertions are equivalent:
\begin{itemize}
\item[-] There is a $C^1_n={C^1_n}^*\ge 0$ satisfying relations \eqref{e6} for all $j, k$.
\item[-] There is a measure $\s\in\cm(\bt)$ generating $C^1_n$ through relations \eqref{e51} for all $j, k$.
\end{itemize}
\end{proposition}

With little changes, the same proposition holds for $\fw_2$, the only difference is that one has to use 
$$\ovl{\lt(\frac 1{1-\bar\a_j t}\rt)}\frac 1{1-\bar\a_k t}=\frac 1{1-\a_j \bar \a_k}\lt(\frac 1{1-\bar\a_k t} + 
\ovl{\frac 1{1-\bar\a_j t}} -1 \rt)
$$
instead of  \eqref{e52}.

For the sake of completeness, we give a criterion for the solvability of the moment problem for $\{c^i_k\}$. Up to inessential details, it is borrowed from \cite[Ch.~3, Sect.~2]{krnu}
\begin{proposition}\label{pr3} The moment problem (see (3'), Theorem \ref{t2}) is solvable iff:
\begin{enumerate}
\item For $\fw_1$:
$$
\lt[\frac 1{1-\a_k\bar\a_j}\big(\r_j\bar\a_kc^1_k+\r_k\a_j\bar c^1_j+c^1_0 (1-|\a_k|^2|\a_j|^2) \big)\rt]_{j,k}\ge 0
$$
\item For $\fw_2$:
$$
\lt[\frac{c^2_k+\bar c^2_j-c^2_0}{1-\a_k\bar\a_j}\rt]_{j,k}\ge 0
$$
\end{enumerate}
\end{proposition}
The remark on the link between $\z_k$ and $1/(1-\bar\a_k t)$ next to \eqref{e4} implies that $c^2_k=\frac 1{\r_k}(\bar\a_k c^1_k+c^1_0)$, so the first claim follow from the second one.

\subsection{ORFs, Blaschke-VG-processes and prediction problems}\label{s32}
Once again, the purpose of this subsection is to specify the content of Sections \ref{s21}, \ref{s22} to the special case of  system $\fw_1$. The construction for $\fw_i,\ i=2, 3$, is completely analogous and we leave it to the interested reader.

So, let $\{\a_k\}_k$ be a sequence of points in $\bd$ satisfying the assumptions of the previous subsection. 
Let $\bx$ be a varying Gaussian process (see the definition in Section \ref{s21})
such that  $C=C_{n,\dots, n+m}=C(X_{n},\dots X_{n+m})$ is a generalized Toeplitz matrix w.r.to $\{w^1_n,\dots, w^1_{n+m}\}$ (a Blaschke-VG-process, to be brief). Proposition \ref{pr2} implies that there exists the unique $\s\in\cm(\bt)$ such that
$$
C=C_{n,\dots, n+m}=\lt[\int_\bt \bar w_{1j} w_{1k} d\s\rt]_{j,k}.
$$
Notice that $C=C^*\ge 0$ and it satisfies relations \eqref{e6}. This means that we can calculate $C_{n,\dots, n+m+1}$ having $C_{n,\dots, n+m}$ and the element $C_{n+m, n+m+1}$. Indeed, just apply \eqref{e6} to recover the last column and the last row of the "new" matrix $C_{n,\dots, n+m+1}$. Hence, to get, say,
$C_{n+l,\dots, n+m+l}$ from  $C_{n,\dots, n+m}$, do the above step-by-step procedure  to come to $C_{n,\dots, n+m+l}$ and then throw away the first $l$ columns and rows.
A counterpart to Theorem \ref{t3} in this situation is as follows.
\begin{proposition}\label{pr4}
Let $\bx$ be a Blaschke-VGP and $\s$ be its spectral measure. Then there exists a unique family of random variables $Z_\xi=Z( . , \xi), \xi\in\bt$ with the properties:
\begin{enumerate}
\item For $\{\xi_1,\dots,\xi_k\}, \xi_j\in\bt, \xi_i\not =\xi_j$, the random variables $\{Z_{\xi_j}\}_j$ are jointly normally distributed.
\item For $I=[\xi_1,\xi_2)\subset \bt$, one writes $Z(I)=Z_{\xi_2}-Z_{\xi_1}$ and
$$
E(|Z(I)|^2)=\s(I),\quad E(Z(I_1)\bar Z(I_2))=0, 
$$
for $I_1\cap I_2=\emptyset$.
\item Finally,
$$
X_n=\int_\bt \cb_n (\xi) dZ( . , \xi).
$$
\end{enumerate}
\end{proposition}

Now,  the generalized orthogonal polynomials $\{\p_k,\p^*_k, \pp_k,\pp^*_k\}$ appearing in Section \ref{s22} are {\it precisely} orthogonal rational functions (denoted by the same letters) studied in \cite{bu1, bu2, bar1}. So, the forward and backward prediction problems from Section \ref{s22} for Blashke-VG-processes are solved {\it precisely} by the ORFs. The asymptotics of  the ORFs $\{\p_n,\p^*_n\}$ describe the behavior of the predictors for $n\to\infty$.  The study of these asymptotics is an interesting (and challenging) analytical problem;  this is the main purpose of the paper Baratchart et al. \cite{bar1}, where, in particular, a special attention is paid to the measures $\s$ from the so-called Szeg\H o class.

In the rest of this subsection we extensively use the terminology and notation of \cite{bar1}; see especially \cite[Sect. 5]{bar1}.
To give a sample of application of results of the work, let us turn back to relations \eqref{e35}.  Recall that $\hat E_n=X_n-\hat X_n$ is precisely one-step ahead prediction error and its "energy" is 
$$E_n=E(|\hat E_n|^2)^{1/2}=E(|X_n-\hat X_n|^2)^{1/2}=\frac 1{|\p^*_n(\a_n)|}.$$  
\begin{proposition}\label{pr5} Suppose that assumptions of \cite{bar1}, Theorem 5.8, hold true. Let $\mathbf{X}$ be a $\fw_1$-VG-process and $\s$ be its spectral measure (lying in the Szeg\H o class). Let $\{\a_{m_n}\}$ be a subsequence of $\{\a_n\}$ and $\lim_{n\to\infty} \alpha_{m_n}=\a\in \bar\bd$.
\begin{itemize}
\item For $\a\in \bd$, \quad
$\lim_{n\to\infty} E_{m_n}=(1-|\a|^2)^{1/2}|S(\a)|.$
\item For $\a\in\bt$, \quad
$\lim_{n\to\infty} E_{m_n}=0.$
\end{itemize}
\end{proposition}
The proposition says, in particular, that in the second case the corresponding subprocess is asymptotically deterministic.


\section{ARMA-processes and their varying Gaussian counterparts}\label{s4}

\subsection{Reminder on "classical" ARMA-processes}\label{ss41}
ARMA-$(p,q)$ (i.e., {\it autoregressive moving average}) processes form a subclass of stationary Gaussian processes. They play an important role in applications and are widely studied. This subsection recalls some basic definitions and results on the topic. Its content is borrowed from Brockwell-Davis \cite[Ch. 3, 4]{bro}, which extensively discusses  different issues pertaining to the subject.

Let $\mathbf{Z}=\{Z_n\}_n$ be a SGP with covariance matrices $\mathbf{C}=\{C_n\}_n$ of the form  $C_n=\mathrm{diag}\, \{\d^2\}$ and zero mean. The process $\mathbf{Z}$ is called a white noise (WN$(\d^2, 0)$, to be brief).

Let $\mathbf{X}$ be a SGP with covariance matrices given by
\begin{equation}\label{e7}
C_n=[c_{k-j}]_{j,k=0,\dots,n}=
\begin{bmatrix}
c_0 & c_1& \dots &c_n\\
\bar c_1 &  \ddots & \ddots & \vdots\\
\vdots & \ddots & c_0 & c_1\\
\bar c_n & \dots &\bar c_1 & c_0
\end{bmatrix}
\end{equation}
where, as always,  $c_{-s}=\bar c_s$.

Let $\cz'$ be a shift operator acting as $\cz' X_n=X_{n-1}$. For polynomials $\p, \t, \deg \p=p, \deg \t=q$, $\mathbf{X}$ is called an ARMA-$(p,q)$ process \cite[Ch. 3, Sect. 1]{bro}, if it satisfies the equation
$$
\p(\cz')X_n=\t(\cz') Z_n,
$$
where $\mathbf{Z}$ is a WN$(\d^2, 0)$. Notice that if the operator $\p(\cz')$ is invertible (in some sense), we can rewrite the above equation formally as $X_n=\p(\cz')^{-1}\t(\cz') Z_n$.

Assuming now that $\psi=\{\psi_j\}\in l^1$, we see that 
$$
Y_n=\psi(\cz') X_n=\sum^\infty_{j=-\infty} \psi_j X_{n-j}
$$
converges in $L^2(dP)$, \cite{bro}, Proposition 3.1.1, and hence is well-defined. Setting $\mathbf{Y}=\{Y_n\}$, we compute readily its covariance matrices \cite{bro}, Proposition 3.1.2,
\begin{equation}\label{e72}
\mathbf{C_Y}=\{\ti C_n\}_n, \  \ti C_n=[\ti c_{k-j}]_{j,k=0, \dots, n}, \quad  \ti c_h=\sum^\infty_{j,k=-\infty}\psi_j\bar \psi_k c_{h-j+k}.
\end{equation}
We see, in particular, that $\mathbf{Y}$ is a GSP. We say that $\mathbf{X}$ is casual, if $X_n=\sum^\infty_{j=0}\psi_j Z_{n-j}$.

\begin{theorem}[{\cite{bro}, Theorem 3.1.1}] Let $\mathbf{X}$ be an ARMA-$(p,q)$ process, $\p(\cz')X_n=\t(\cz')Z_n$ and polynomials $\p, \t$ do not have common zeroes. Then $\mathbf{X}$ is casual {\it iff} the zeroes of $\p$ lie in $\{|z|>1\}$. In this case,
\begin{equation}\label{e71}
X_n=\psi(\cz') Z_n=\sum^\infty_{j=-\infty} \psi_j Z_{n-j},
\end{equation}
where 
$$
\psi(z)=\frac{\t(z)}{\p(z)}=\sum^{\infty}_{j=-\infty} \psi_j z^j
$$
for $|z|\le 1$.
\end{theorem}

Furthermore, the spectral measure for the above process $\mathbf{X}$ can be easily  written down, \cite{bro}, Theorem 4.4.2,
\begin{equation}\label{e74}
d\s_{\mathbf{X}}(t)= \d^2 \frac{|\t(\bar t)|^2}{|\p(\bar t)|^2}\, dm(t),
\end{equation}
where $t=e^{i\phi}\in \bt$ and $dm(t)=\frac{dt}{2\pi i t}=\frac 1{2\pi} d\phi$ is the normalized Lebesgue measure on the unit circle. Moreover, if $Z_n=\int_\bt \xi^n dZ_{\mathbf{Z}}(.,\xi)$, we obtain the spectral decomposition for $\mathbf{X}$
\begin{equation}\label{e75}
 X_n=\int_\bt \xi^n \  \d\frac{\t(\bar \xi)}{\p(\bar \xi)}\, dZ_{\mathbf{Z}}(.,\xi),
\end{equation}
see Theorem \ref{t03}.

\subsection{Varying analogues of ARMA-type processes}\label{ss42}
Here, we rewrite the results of the previous subsection for varying Gaussian processes as introduced in Section \ref{s2}. We think especially of the VG-processes satisfying formula \eqref{e71} w.r.to a white noise $\mathbf{Z}$. Of course, there are no mathematical reasons to assume that the coefficients $\{\psi_j\}$ come from a rational function, but since the situation is important in practice, we keep it in mind and do some  comments on the issue. The corresponding VG-processes will be called varying ARMA-type processes ($\fw$-VARMA- or VARMA-processes, to be brief).

To start with,  let  $\fw$ be a general system satisfying usual properties (see the beginning of Section \ref{s1}, p.~\pageref{p1}). For a VG-process $\mathbf{X}$ we always assume \eqref{e330} and that the operator $\cz'=\cz^{-1}$, $\cz$ being defined right after the formulation of Theorem \ref{t3}, is bounded, i.e. $||\cz'||\le R$ for some $R\ge 1$. Sometimes we require that $||\cz'^{-1}||\le R$, too. These assumptions on $\cz'$ ($\cz'^{-1}$) can be easily verified, for example, for some systems of ORFs (i.e., $\{\a_k\}$ compactly contained in $\bd$ and $c_1 m\le \s\le c_2 m$ on $\bt$ with  $c_1,c_2>0$).

For instance, we may say that $\mathbf{X}$ is a varying white noise process (VWN), if $d\s=dm$, $\s$ being its spectral measure, see Theorem \ref{t2} and the discussion at the beginning of Section \ref{s2}.  

 Let $\p,\t$ be polynomials of degree $p$ and $q$, respectively. We say that $\mathbf{Y}$ is VARMA-process w.r.to  a $\fw$-VG-process $\mathbf{X}$, if $\p(\cz') Y_n=\t(\cz') X_n$.
More generally, a process $\mathbf{Y}$ is $\psi$-VARMA-process w.r.to $\mathbf{X}$, if
\begin{equation}\label{e73}
 Y_n=\sum^\infty_{j=-\infty} \psi_j X_{n-j},
\end{equation}
where $\psi=\{\psi_j\}$ is so that
\begin{equation}\label{e76}
\sum^\infty_{j=-\infty}|\psi_j|R^{|j|}<\infty.
\end{equation}
Notice that $\mathbf{Y}$ is not a $\fw$-VG-process in general, and the results of Section \ref{s2} do not apply. Nevertheless, we are able to obtain conclusions similar to Theorems \ref{t2}, \ref{t3} for these processes with the help of filtering tricks \eqref{e302}-\eqref{e304} starting from the "reference" VG-processes $\mathbf{X}$. We also note that the introduced classes of VARMA- ($\psi$-VARMA-) processes are exactly {\it properly filtered} VG-processes. The filters are of course linear and stationary. 

It is plain that conditions \eqref{e330} and \eqref{e76} imply that the sum \eqref{e73} converges in $L^2(dP)$ and $\mathbf{Y}$ is well-defined.  As before, we say that $\mathbf{Y}$ is casual w.r.to $\mathbf{X}$, if $Y_n=\sum^\infty_{j=0} \psi_j X_{n-j}$.

The covariance matrices $\mathbf{C_Y}=\{C_{n, \mathbf{Y}}\}$ are easy to compute. Indeed, one has
$$
\begin{bmatrix}
\dots & Y_{-1}& Y_0 &Y_1 &\dots
\end{bmatrix}
=
\begin{bmatrix}
\dots&X_{-1}&X_0&X_1&\dots
\end{bmatrix} \
\begin{bmatrix}
\ddots &   &\ddots &   & \ddots \\
\ddots & \psi_0 & \psi_{1}& \psi_{2} & \ddots\\
\ddots & \psi_{-1} & \psi_0 & \psi_{1} & \ddots\\
\ddots  & \psi_{-2} & \psi_{-1} & \psi_{0} & \ddots\\
\ddots &  &\ddots &  & \ddots 
\end{bmatrix} \
$$
and, with $\Psi$ denoting the above matrix, 
$$
C_{n, \mathbf{Y}}=\Psi^* C_{n, \mathbf{X}} \Psi, 
$$ 
compare to \eqref{e72}.

\begin{proposition}\label{pr7} Let $\mathbf{Y}$ be a VARMA-process w.r.to a $\fw$-VG-process $\mathbf{X}$, and $||\cz'||\le R$. Suppose that polynomials $\p$, $\t$ do not have common zeroes. Put
$$
\psi(z)=\frac{\t(z)}{\p(z)}=\sum^\infty_{j=-\infty} \psi_j z^j.
$$
If the zeroes of $\p$ are in $\{|z|>R\}$, $\mathbf{Y}$ is casual w.r.to $\mathbf{X}$ and $Y_n=
\sum^\infty_{j=0} \psi_j X_{n-j}$.
\end{proposition}

Turning to the spectral part of the matter, we see that spectral characteristics of $\mathbf{Y}$ defined by \eqref{e73}, are readily expressible in terms of spectral parameters of $\mathbf{X}$. In fact, one has
$$
E(Y_k\bar Y_j) =\int_\bt \lt(\sum^\infty_{s=-\infty} \psi_s w_{k-s}\rt) \lt(\sum^\infty_{p=-\infty} \psi_p w_{j-p}\rt)\, d\s,
$$
where $E(X_k\bar X_j)=\int_\bt w_k\bar w_j\, d\s$ and $\s$ is a spectral measure of $\mathbf{X}$. Furthermore,
$$
Y_n=\int_\bt \lt(\sum^\infty_{j=-\infty} \psi_j w_{n-j}\rt) (\xi) dZ(., \xi),
$$
where $X_n=\int_\bt w_n(\xi) dZ(.,\xi)$ is the spectral representation for  $\mathbf{X}$. It is instructive to compare the last displayed formulas to \eqref{e74} and \eqref{e75}.

\end{document}